\documentclass[11pt,reqno]{amsart}
\usepackage{graphicx}
\usepackage{color}
\usepackage{amsmath,amssymb,amsthm,amsfonts,mathrsfs,stmaryrd}
\usepackage{cases,indentfirst}
\usepackage{float}
\usepackage{multicol}
\usepackage{hyperref}

\usepackage{amsmath, amsfonts, amsthm, amssymb,mathrsfs, amscd, graphicx, cite, color}

\makeatletter

\newtheorem{theorem}{Theorem}[section]
\newtheorem{lemma}{Lemma}[section]
\newtheorem{proposition}{Proposition}[section]

\theoremstyle{definition}
\newtheorem{definition}{Definition}[section]

\theoremstyle{remark}
\newtheorem{remark}{Remark}[section]

\newcommand{\na}{\nabla}
\newcommand{\n}{\rho}
\newcommand{\de}{\delta}
\newcommand{\om}{\Omega}
\newcommand{\la}{\label}

\newcommand{\bnn}{\begin{eqnarray*}}
\newcommand{\enn}{\end{eqnarray*}}
\newcommand{\ba}{\begin{aligned}}
\newcommand{\ea}{\end{aligned}}
\newcommand{\be}{\begin{equation}}
\newcommand{\ee}{\end{equation}}

\def\p{\partial}
\def\norm[#1]#2{\|#2\|_{#1}}
\def\g{\gamma}

\def\a{\alpha}

\def\r{\mathbb{R}^{3}}
\def\lam{\lambda}
\def\th{\theta}
\def\de{\delta}
\def\lap{\triangle}

\renewcommand{\div}{{\rm div}}

\theoremstyle{definition}
\numberwithin{equation}{section}

\begin{document}

\title[Weak  solutions to  CH/NS for compressible fluids]
{\bf Weak solutions to the stationary Cahn-Hillard/Navier-Stokes equations for compressible fluids}


\author[Z. Liang]{Zhilei Liang}
\address{School of  Mathematics, Southwestern
University of Finance and Economics,  Chengdu 611130,  China.}
\email{zhilei0592@gmail.com}

\author[D. Wang]{Dehua Wang*}\thanks{* Corresponding author.}
\address{Department of Mathematics, University of Pittsburgh,
                           Pittsburgh, PA 15260, USA.}
\email{dwang@math.pitt.edu}

\begin{abstract}
We are concerned with the   Cahn-Hilliard/Navier-Stokes equations for the stationary compressible flows   in a three-dimensional bounded domain. The governing  equations consist  of  the stationary Navier-Stokes equations describing the  compressible fluid  flows  and the stationary  Cahn-Hilliard type diffuse  equation   for  the mass  concentration difference. We prove  the existence  of weak solutions when   the  adiabatic exponent $\g$ satisfies $\g>\frac{4}{3}$. The  proof is based on the weighted total energy estimates and the new techniques developed  to overcome the  difficulties from the  capillary  stress.   
 \end{abstract}

\date{\today}

\maketitle

\section{Introduction}

The Cahn-Hilliard/Navier-Stokes    system  is one of the  important diffuse interface models (cf.\cite{am,chan,low})
 describing  the evolution of mixing fluids.
 The mixture is  assumed to be macroscopically immiscible,  with a partial mixing in a small interfacial region where the sharp interface
 is regularized  by the Cahn-Hilliard   type diffusion   in terms   of  the mass  concentration difference. Roughly speaking,
 the Cahn-Hilliard  equation is  used for modeling   the loss of mixture homogeneity and the formation of pure phase regions, while
 the Navier-Stokes   equations describe  the hydrodynamics  of the mixture that is  influenced by  the order parameter, due to the surface  tension and its   variations,  through an extra capillarity force  term.

In this paper,  we are interested in  the   following stationary  Cahn-Hilliard/Navier-Stokes  system  for  the    mixture of compressible  fluid  flows  in a three-dimensional 
bounded domain $\om\subset\mathbb{R}^3$:
  \begin{equation}\label{1}
\left\{\ba
&\div(\n u)=0,\\
&\div(\n u\otimes u) = \div\left(\mathbb{S}_{ns}+\mathbb{S}_{c}-P\mathbb{I} \right)+\n g,\\
&\div (\n u c)=\lap \mu,\\
&\n\mu=\n \frac{\p f(\n,c)}{\p c}-\lap c,
\ea \right.
\end{equation}
 where  $\n$ denotes the total density,  
 $u$ the mean velocity field, $c$  the mass concentration  difference of the two components,  $\mu$  the chemical potential,  
  and $g$   the external force; the tensor
   \be\la{c0} \mathbb{S}_{ns} =\lambda_{1}\left(\na u+(\na u)^{\top}\right)+\lambda_{2} \div u\mathbb{I},\ee
   is the    Navier-Stokes    stress tensor,
where  
$\mathbb{I}$ is   the $3\times 3$ identity matrix,
   $\lambda_{1},\,\lambda_{2}$ are     constant  such that 
\be\la{01}\lambda_{1}>0,\quad2\lambda_{1}+3\lambda_{2}\ge0;\ee
the tensor
  \be\la{b0}\mathbb{S}_{c}=-\na c\otimes \na c+\frac{1}{2}|\na c|^{2}\mathbb{I},\ee 
 is the  capillary  stress tensor;  and 
  \be\la{b00} P=\n^{2}\frac{\p f(\n,c)}{\p \n} ,\ee  
is  the  pressure       with   the free energy density (cf. \cite{Fei,low})
\be\ba\la{b1} f(\n,c) 
=\n^{\g-1}+H_{1}(c)\ln \n  +H_{2}(c),\ea\ee
where   $\g>1$ is the adiabatic exponent,    and $H_{i}\,(i=1,2)$ are two given functions.
The corresponding  evolutionary  diffuse interface model  was derived in \cite[Section 2.2]{Fei} where the existence of weak solutions was obtained for $\g>\frac{3}{2}$.
 We refer the readers to \cite{am,chan,low,Fei,liang} for more discussions on the physics and models of mixing fluids with   diffuse interfaces.


We briefly review some related results in literature. 
For the stationary  Navier-Stokes equations of compressible flows, the existence of weak solutions was studied in Lions \cite{lion} with $\g>\frac{5}{3}$, 
Novotn\'{y}-Stra\v{s}craba \cite{novo1} with $\g>\frac{3}{2}$, 
Frehse-Steinhauer-Weigant \cite{freh} with $\g>\frac{4}{3}$,  Plotnikov-Weigant\cite{pw} with $\g>1$,
as well as in Jiang-Zhou \cite{jiang} and  Bresch-Burtea  \cite{Bresch-Burtea} for  periodic domains.
For the stationary Cahn-Hilliard/Navier-Stokes equations of incompressible flows,    the  existence  of weak solutions was obtained  in  
Biswas-Dharmatti-Mahendranath-Mohan \cite{bis}, Ko-Pustejovska-Suli \cite{ko1}, and Ko-Suli \cite{ko2}.
For  the compressible  Cahn-Hilliard/Navier-Stokes equations, Liang-Wang in \cite{liang} proved the  existence  of weak solutions  in case of  the adiabatic exponent 
$\g>2.$  See \cite{lion,novo1,fei,freh,pw,jiang,mpm,mp1,bis,ko1,ko2,liang,Bresch-Burtea,bis2021,CJWZ2020} and  their references for more results.

%
%

In  this    paper, we shall continue our study on the    existence  
 of weak solutions,    and   improve our previous  result  obtained in \cite{liang}  for $\g>2$ to the case of   $\g>\frac{4}{3}$ for the  stationary equations \eqref{1} subject to 
  the  following  boundary conditions:
\be\la{1a}u=0,\,\,\,\frac{\p c}{\p n}=0,\,\,\,\frac{\p \mu}{\p n}=0,\quad {\rm on} \,\,\,\p\om,\ee
and the additional conditions:
\be\la{0r} \int \n(x)dx=m_{1}>0,\quad    \int \n(x) c(x)dx=m_{2},\ee
with two given   constants $m_1$ and $m_{2}$, where $n$ is the normal vector of   $\p\om$.

Before stating our main results, we  introduce some notation  that will be used throughout this paper.
 For two given matrices  $\mathbb{A}=(a_{ij})_{3\times3}$ and $\mathbb{B}=(b_{ij})_{3\times3}$,   we denote their   scalar product   by  $\mathbb{A}:\mathbb{B}=\sum_{i,j=1}^{3}a_{ij}b_{ij}$.  For two vectors $a,\,b\in \mathbb{R}^{3}$,  
denote $a\otimes b =(a_{i}b_{j})_{3\times3}.$
We use $ \int f=\int_{\om}f(x)dx$ for simplicity.
For  any $p\in [1,\infty]$  and integer $k\ge 0,$       
$W^{k,p}(\om)$    is   the standard  Sobolev space  (cf. \cite{ad}),
and \bnn 
&W_{0}^{k,p} =\left\{f\in W^{k,p}:  \ f|_{\p\om}=0\right\},\quad  W_{n}^{k,p} =\left\{f\in W^{k,p}:   \   \frac{\p f} {\p n}|_{\p\om}=0\right\},\\
& L^{p}=W^{0,p},\quad H^{k}=W^{k,2},\quad H_0^k=W_0^{k,2}, \quad H_n^k=W_n^{k,2}, \\
&\overline{L^{p}}=\big\{f\in L^{p}: \ (f)_{\om}=0\big\},
\enn
where  $(f)_{\om}=\frac{1}{|\om|}\int f$ is  the average of $f$ over $\om$.

As in \cite{liang}, we define the weak solution as follows.
\begin{definition}\la{defi}
 The vector of functions $(\n,u,\mu,c)$ is called    a weak solution to the problem  \eqref{1}-\eqref{0r}, 
 if
\bnn\ba &  \n\in L^{\g+\th}(\om ),\,\,\,\n\ge0\,\,{\rm a.e. \ in}\,\,\om,
\quad u\in H_{0}^{1}(\om),\quad \mu\in  H_{n}^{1}(\om), \quad c\in W_{n}^{2,p}(\om),\ea\enn
  for  some $p>\frac{6}{5}$ and  $\th>0$,  and the following properties hold true:  

(i)  The system   $\eqref{1}$  is   satisfied in the sense of distributions in $\Omega$,
and \eqref{0r}  holds for the given constants $m_{1}>0$ and $m_{2}\in \mathbb{R}.$

(ii)
If $(\n,u)$ is prolonged by zero outside $\Omega$, then both the equation $\eqref{1}_{1}$ and
$$\div(b(\n) u)+ \left(b'(\n)\n-b(\n)\right)\div u=0$$ are satisfied in the sense of distributions  in $\r$,
where  $b\in C^{1}([0,\infty))$ with  $b'(z)=0$ if $z$ is  large enough.

 (iii) The   following energy inequality  is valid:
\bnn \int \left(\lambda_{1}|\na u|^{2}+(\lambda_{1}+\lambda_{2})(\div u)^{2}+|\na \mu|^{2}\right) dx
\le \int  \n g \cdot u.\enn
\end{definition}

\bigskip

We now state our main result.
\begin{theorem}\la{thm1.1}
Let  $\om\subset\mathbb{R}^3$ be a bounded domain with  $C^{2}$  boundary.      Assume that \be\la{0}\g>\frac{4}{3}, \ee
and  
\begin{equation}\label{b2} g  \in  L^{\infty}(\om),\quad |H_{i}(c)|+|H_{i}'(c)|\le \overline{H} \quad \forall \, c\in\mathbb{R}, \quad i=1,2,\end{equation}
for some constant $\overline{H}<\infty$.
Then, for any given constants $m_{1}>0$ and $m_{2}$,  the problem \eqref{1}-\eqref{0r}  admits a  weak solution  $(\n,u,\mu,c)$  in the sense of Definition \ref{defi}.
\end{theorem}
The   main  contribution of this paper  is to develop new ideas to improve  the  existence result of \cite{liang}   from the   adiabatic exponent $\g>2$ in \cite{liang}  to a wider range  $\g>\frac{4}{3}$. 
\textcolor{black}{Our approach is mainly  motivated by  the papers \cite{jiang,pw} where the authors studied the existence of weak solutions to the stationary Navier-Stokes equations of compressible fluids.} In order to prove the Theorem \ref{thm1.1},   we    start with    the approximate solution sequence $(\n_{\de},u_{\de},\mu_{\de},c_{\de})$
stated in   Proposition \ref{pro} in Section 2,  and use  the weighted total energy as in  \cite{jiang,pw}  together with   new techniques to handle the capillary  stress to establish  the  uniform in $\de$ bound on  $(\n_{\de},u_{\de},\mu_{\de},c_{\de})$ in \eqref{y16}.  Then, we shall be able to take the limit as $\de\rightarrow0$ and  complete the proof of Theorem \ref{thm1.1} by means of the weak convergence arguments in \cite{liang}.
More precisely, our proof includes the following key ingredients and new ideas:
\begin{enumerate}
  \item In light of    \cite{jiang,pw},  {\color{black}  for any given $x^{*}\in \overline{\om}$} we estimate   the weighted   total energy
   $$\int_{\om} \frac{\left(\de \n^{4}+P+\n|u|^{2}\right)(x)}{|x-x^{*}|^{\a}}dx,$$
   instead of $$\int_{\om} \frac{\left(\de \n^{4}+P\right)(x)}{|x-x^{*}|^{\a}}dx,$$
where the advantage is that   the involved kinetic energy    $$\int_{\om} \frac{\n|u|^{2}(x)}{|x-x^{*}|^{\a}}dx$$  
helps  us   relax the    restriction on $\g$.


  \item  In order to analyze the weighted total energy we need to overcome the new  difficulties caused by  the  capillary  stress  $\mathbb{S}_{c}$ in \eqref{b0}, besides the Navier-Stokes    stress tensor $\mathbb{S}_{ns}$.  
    In particular,  we are required  to control   $\|\n \mu  \|_{L^{\frac{3}{2}}}^{2}$
appearing in  \eqref{r5-1} and \eqref{r5-2}. For this purpose  we make the following estimate
\bnn\ba & \int_{\om\cap B_{r_{0}}(x^{*})} \frac{\left(\de \n^{4}+P+\n|u|^{2}\right)(x)}{|x-x^{*}|^{\a^{2}}}dx\\
&\le C r_{0}^{\a(1-\a)} \int_{\om\cap B_{r_{0}}(x^{*})} \frac{\left(\de \n^{4}+P+\n|u|^{2}\right)(x)}{|x-x^{*}|^{\a}}dx,\ea\enn where  $r_{0}>0$ is small and  $\a\in (0,1)$.
\textcolor{black}{By virtue of  the Finite Coverage Theorem, 
$\overline{\om}$ can be covered by a finite number of balls of radius $r_0$ centered at $x_{1}^{*},...,x_{K}^{*}$, then}
\bnn\ba&\sup_{x^{*}\in \overline{\om}}\int_{\om} \frac{\left(\de \n^{4}+P+\n|u|^{2}\right)(x)}{|x-x^{*}|^{\a^{2}}}dx\\
 &\le \max_{1\le k\le K}\int_{\om\cap B_{r_{0}}(x_{k}^{*})} \frac{\left(\de \n^{4}+P+\n|u|^{2}\right)(x)}{|x-x^{*}|^{\a^{2}}}dx\\
 &\le Cr_{0}^{\a(1-\a)}\|\n \mu  \|_{L^{\frac{3}{2}}}^{2}+\cdots. \ea\enn
Next, we assume   the following   {\it a priori} bound \be\la{j1}\textbf{M}=\max\{1, \|\n\|_{L^{2}}\}<\infty\ee that is uniform in   $\de>0.$ If we select  $r_{0}=r_{0}(\a,\textbf{M})$   small enough such that
\bnn r_{0}^{\a(1-\a)}\|\n \mu  \|_{L^{\frac{3}{2}}}^{2}
\le C r_{0}^{\a(1-\a)}\textbf{M}^{\frac{8}{3}}(\|\na \mu \|_{L^{2}}^{2}+1)\le C(\|\na \mu \|_{L^{2}}^{2}+1),\enn
we are able to derive the following estimate
\bnn\ba\int_{\om} \frac{\left(\de \n^{4}+P+\n|u|^{2}\right)(x)}{|x-x^{*}|^{\a^{2}}}dx\le C+C \|\na \mu  \|_{L^{2}}^{2}+\cdots.\ea\enn
\item With the above two key steps, we can show that there is a constant $C$ that does not rely on  \textbf{M}, such that
 \bnn \|\n^{\g}\|_{L^{s}}\le C+C\|\n \mu  \|_{L^{\frac{3}{2}}}^{2} \le C+C\|\n  \|_{L^{2}}^{\frac{4}{3}}\le C+\frac{1}{2}\|\n^{\g}  \|_{L^{s}},\enn as long as $\g s>2$.
  This yields    $\|\n^{\g}\|_{L^{s}}\le 2C$, and then we have the estimate $\|\n\|_{L^{2}}\le C_0$ for some positive constant $C_0$ independent of \textbf{M}.  By choosing the {\it a priori} bound  $\textbf{M}=2C_0$,  one can close the {\it a priori}  assumption \eqref{j1} and prove the existence of weak solutions in the Theorem \ref{thm1.1}.
\end{enumerate}

The rest of the paper is organized as follows. 
In Section 2, we present the approximate solutions constructed in \cite{liang} and provide some preliminary lemmas.
In Section 3, we prove the Theorem   1.1.

\bigskip

\section{ Approximate Solutions and Preliminaries}

We start with  the following approximate solutions constructed in \cite{liang}.
\begin{proposition}[Theorem 4.1, \cite{liang}] 
\la{pro} 
Under the  assumptions of  Theorem \ref{thm1.1},  for any fixed parameter $\de>0$  and   any given constants $m_{1}>0$ and $m_{2}$,   the  system
 \begin{equation}\label{n6}
\left\{\ba
&\div (\n u)=0,\\
&\div(\n u\otimes u)+\na \left(\de \n^{4} + \n^{2} \frac{\p f}{\p \n}\right)
 =\div\left(\mathbb{S}_{ns}+\mathbb{S}_{c}\right)+\n g,\\
 &\div(\n u   c)=\lap \mu,\\
&\n\mu= \n \frac{\p f}{\p c}-\lap c,\ea \right.
\end{equation}
with  the boundary conditions \eqref{1a},  admits a weak solution $(\n_{\de},u_{\de},\mu_{\de},c_{\de})$  in the sense of distributions  such that
\be\la{e7b}\|\n_{\de}\|_{L^{1}} =m_{1}, \quad\int \n_{\de} c_{\de} =m_{2},\ee
 \be\ba\la{e7c} \n_{\de}\in L^{5}(\om),\,\,\,\n_{\de}\ge0 \text{ a.e. in } \om,\,\,\,u_{\de}\in  H_{0}^{1}(\om),\,\,\,(\mu_{\de},\,c_{\de})\in H_{n}^{1}(\om)\times H_{n}^{1}(\om),\ea\ee
and
\be\la{e7a}\int \left(\lambda_{1}|\na u_{\de}|^{2}+(\lambda_{1}+\lambda_{2})(\div u_{\de})^{2}+|\na \mu_{\de}|^{2}\right)\le \int  \n_{\de} g  \cdot u_{\de}.\ee
\end{proposition}

\begin{lemma}\la{lem2.1} 
Let $(\n_{\de},u_{\de},\mu_{\de},c_{\de})$  be the solution in Proposition \ref{pro}. Then we have
 \be\la{3.30}\ba
 \|\mu_{\de}\|_{L^{p}}\le C\left(1+\|\na \mu_{\de} \|_{L^{2}})(1+  \|\n_{\de}\|_{L^{\frac{6}{5}}}\right),\quad p\in [1,6],\ea\ee
 where the constant $C$ is independent of $\de.$
 \end{lemma}
 
\begin{proof}
Thanks to \eqref{b1}, \eqref{b2}, the boundary conditions \eqref{1a},    one has,  from  $\eqref{n6}_{4}$,
 \be\la{bb9}  
 \int  \n_{\de} \mu_{\de}= \int  \left(\n_{\de} \frac{\p f}{\p c_{\de}}+ \lap c_{\de}\right)=\int \n_{\de} \frac{\p f}{\p c_{\de}}  \le C \left( \|\n_{\de}\ln\n_{\de}\|_{L^{1}} +1\right).\ee
Using  \eqref{bb9} together with  \eqref{e7b} and the embedding inequality guarantees that
 \bnn\ba
 \int \mu_{\de}
 &=\frac{|\om|}{m_{1}}\int \n_{\de}\left(\mu_{\de}\right)_{\om}\\
&=\frac{|\om|}{m_{1}}\int \n_{\de} \mu_{\de}-\frac{|\om|}{m_{1}}\int \n\left(\mu_{\de}-\left(\mu_{\de}\right)_{\om}\right)\\
&\le C \left( \|\n_{\de}\ln\n_{\de}\|_{L^{1}} +1\right)+C \|\n_{\de}\|_{L^{\frac{6}{5}}}\|\na \mu_{\de}\|_{L^{2}} ,
\ea\enn
which implies  
 \be\ba\la{bb12}
 \|\mu_{\de}\|_{L^{1}}
&\le C \|\na \mu_{\de} \|_{L^{2}} + C \left( \|\n_{\de}\ln\n_{\de}\|_{L^{1}} +1\right)+C \|\n_{\de}\|_{L^{\frac{6}{5}}}\|\na \mu_{\de}\|_{L^{2}}\\
&\le C(1+\|\na \mu_{\de} \|_{L^{2}})(1 + \|\n_{\de}\|_{L^{\frac{6}{5}}}).
\ea\ee
From \eqref{bb12} and the interpolation inequality, we obtain   \eqref{3.30}. The proof of Lemma \ref{lem2.1} is completed.
\end{proof}

The  next  lemma  gives   an embedding  from $H^{1}$ to $L^{2}$ in a three-dimensional bounded domain, via    the Green representation formula.
\begin{lemma} \la{lem2.3}  
Let $\om\subset \r$ be a bounded domain with $C^{2}$ boundary and $f\in L^{2}(\om)$ satisfy
\bnn f\ge0 \quad{\rm and}\quad \int_{\om}\frac{f(x)}{|x-x^{*}|}dx\le \mathbb{E}, \quad \forall \,\,x^{*}\in \om,\enn
for some constant $\mathbb{E}>0$. 
Then, there is a constant $C$ which depends only on $\om$, such that

{\rm (i)} If $u\in H_{0}^{1}(\om),$ then
\be\la{y18} \int_{\om}|u|^{2} fdx\le C\mathbb{E} \|u\|_{H_{0}^{1}(\om)}^{2}.\ee

{\rm (ii)} If $\mu\in H_{n}^{1}(\om)$ and $(f)_{\om}=0$, then \be\la{y19} \int_{\om}\mu^{2} fdx\le C\mathbb{E} \|\na \mu\|_{L^{2}(\om)}^{2}.\ee
\end{lemma}
\begin{proof}
The proof of the case (i) can be found in \cite[Lemma 4]{pw}. Here  we  prove the case (ii). Let $H$ be a solution to  the Neumann  boundary value problem:
\be\la{z1}  \lap H =f\in \overline{L^{2}}\,\,\,{\rm in}\,\,\, \om,\quad{\rm with}\quad
 \frac{\p H}{\p n}=0\,\,\,  {\rm on}\,\,\, \p\om. \ee
Recalling the  Green    representation   formula $H(x^{*})=\int_{\om}G(x^{*},x)f(x)dx$,  we have
 \be\la{z3}\ba \|H\|_{L^{\infty}}&\le C\sup_{x^{*}\in \om}\int_{\om}\frac{f(x)}{|x-x^{*}|}dx\le C\mathbb{E}.\ea\ee
 Thanks to  \eqref{z1},  using integration by  parts yields  
\be\la{ip}\ba \int \mu^{2} f&= \int \mu^{2}\lap H=-2\int \mu\na \mu\cdot\na  H\le 2\|\na \mu\|_{L^{2}}\left(\int \mu^{2}|\na H|^{2}\right)^{\frac{1}{2}}. \ea\ee
From \eqref{ip} we then derive the following estimate:
  \bnn\ba \int \mu^{2}|\na H|^{2}&=-\int \mu^{2}H\lap H-2\int\mu\na \mu H\na H\\
 &\le \|H\|_{L^{\infty}}  \int |\mu|^{2}f+2\|H\|_{L^{\infty}}\|\na \mu\|_{L^{2}}\left(\int |\mu|^{2}|\na H|^{2}\right)^{\frac{1}{2}} \\
 &\le 4\|H\|_{L^{\infty}}\|\na \mu\|_{L^{2}}\left(\int |\mu|^{2}|\na H|^{2}\right)^{\frac{1}{2}},\ea\enn
which implies 
 \bnn\ba \left(\int |\mu|^{2}|\na H|^{2}\right)^{\frac{1}{2}}\le 4\|H\|_{L^{\infty}}\|\na \mu\|_{L^{2}}.\ea\enn
 Substituting the above inequality into \eqref{ip} gives that  
\bnn  \| \mu^{2}f\|_{L^{1}} \le 8\|H\|_{L^{\infty}}\|\na \mu \|_{L^{2}}^{2}.\enn
Then  \eqref{y19} follows from \eqref{z3}.  The proof of Lemma \ref{lem2.3} is completed.
\end{proof}

Finally, we present the properties of the Bogovskii operator
 whose proof is available  in \cite{ga,novo}.
\begin{lemma}[Bogovskii]\la{lem2.2}
Let $\om$ be a bounded Lipschitz domain. There is a linear operator 
$\mathcal{B}= (\mathcal{B}^{1},\mathcal{B}^{2},\mathcal{B}^{3}): \overline{L^{p}}\to W_{0}^{1,p}$ for $p\in (1,\infty)$, such that, for $f\in\overline{L^{p}}$,


{\rm (i)} \bnn \div \mathcal{B}(f)=f\,\,a.e.\,\,{\rm in}\,\,\om, \enn 

{\rm (ii)}  \bnn \|\na \mathcal{B}(f)\|_{L^{p}}\le C(p,\om)\|f\|_{L^{p}}. \enn 
\end{lemma}


\bigskip

\section{Proof  of Theorem \ref{thm1.1}}

For the approximate  solution
$(\n_{\de},u_{\de},\mu_{\de},c_{\de})$  given  in  Proposition \ref{pro},  if we can show that there is a constant $C$  uniform in $\de$  such that
\be\la{y16} \|\de \n_{\de}^{4} + \n^{\g}\|_{L^{s}}+\|u_{\de}\|_{H_{0}^{1}}+
\|\mu_{\de}\|_{H_{n}^{1}}+\|c_{\de}\|_{W_{n}^{2,\frac{3}{2}}}\le C,\quad \g s>2,\ee
\textcolor{black}{then, from  \eqref{y16}  we are able to control the possible  oscillation of density and the nonlinearity in the free energy density \eqref{b1}, and hence we can  take the limit as $\de\rightarrow 0$ to   prove  that 
the  approximate solution $(\n_{\de},u_{\de}, \mu_{\de},c_{\de})$ 
converges weakly to  some limit function  which satisfies    \eqref{1}-\eqref{0r} in the sense of Definition \ref{defi}.  This convergence proof relies heavily on  the compactness  arguments  in  \cite{fei,liang,lion,novo} and the details can be found  in  \cite{liang}.}  Therefore,   it suffices  to   prove the following proposition  in order to complete  the proof of  Theorem \ref{thm1.1}.

\begin{proposition}\la{thm1.2} Let  the  assumptions in Theorem \ref{thm1.1} hold  true. Assume that
   \be\la{y10}  \frac{2}{\g}<s\le \frac{3}{2} \quad {\rm and}\quad \frac{4}{3}<\g\le 2.\ee
   Then, the  solutions $(\n_{\de},u_{\de},\mu_{\de},c_{\de})$ stated in Proposition \ref{pro} satisfy \eqref{y16}.
\end{proposition}

\begin{remark} In   case when  $\g>2$, the  existence of weak solutions to the problem \eqref{1}-\eqref{0r} has been established
 in \cite{liang}. 
 \end{remark}

For the sake of simplicity of notation, in the proof of the Proposition \ref{thm1.2}  below we will drop the subscript in  $(\n_{\de},u_{\de},\mu_{\de},c_{\de})$ and  denote it      by $(\n,u,\mu,c)$.

\begin{lemma}\la{lem3.1} 
Under   the   assumptions of Proposition \ref{thm1.2}, we have
\be\ba\la{q1a} 
\left\|\de \n^{4}+\n^{2} \frac{\p f}{\p \n}\right\|_{L^{s}} \le C \left(1 +\|\na u\|_{L^{2}}+ \|\n |u|^{2}\|_{L^{s}}+ \|\n \mu\|_{L^{\frac{6s}{3+2s}}}^{2}\right),\ea\ee where $s$ is defined in \eqref{y10}.  Here and below, the capital letter $C>0$ denotes a generic constant which may rely on
$m_{1},m_{2},\g,\overline{H},\lam_{1},\lam_{2}, |\om|,\|g\|_{L^{\infty}}$
 but   is  independent of  $\de.$\end{lemma}

\begin{proof}
For any $h\in L^{\frac{s}{s-1}}$,  we test the equation  $\eqref{n6}_{2}$ against $\mathcal{B}(h-(h)_{\om})$ and  deduce that
\be\la{q5.3}\ba  & \int\left(\de \n^{4}+\n^{2} \frac{\p f}{\p \n}\right) h\\
&= \left(h\right)_{\om}\int \left(\de \n^{4}
+\n^{2} \frac{\p f}{\p \n}\right) -\int \n g \cdot \mathcal{B}\left(h-(h)_{\om}\right)+ \int \mathbb{S}_{ns}:\na \mathcal{B}(h-(h)_{\om})\\
&\quad -\int \n u\otimes u:\na \mathcal{B}(h-(h)_{\om})+\int \mathbb{S}_{c} :\na \mathcal{B}(h-(h)_{\om})\\
&\le C\|h\|_{L^{\frac{s}{s-1}}}\left(1+\|\de \n^{4}+\n^{2} \frac{\p f}{\p \n}\|_{L^{1}}+\|\na u\|_{L^{2}}+ \|\n |u|^{2}\|_{L^{s}}+ \|\na c\|_{L^{2s}}^{2}\right),
\ea\ee
where  we have used \eqref{e7b}, Lemma \ref{lem2.2},  the fact $\frac{s}{s-1}>2,$ and the following inequality:
\bnn \ba&\|h\|_{L^{2}}+\|\mathcal{B}(h-(h)_{\om})\|_{L^{\infty}}+\|\na\mathcal{B}(h-(h)_{\om})\|_{L^{2}}
+\|\na\mathcal{B}(h-(h)_{\om})\|_{L^{\frac{s}{s-1}}}\\&\le C\|h\|_{L^{\frac{s}{s-1}}}.\ea\enn
Now we choose 
\bnn h=\left(\frac{|\de\n^{4}+\n^{2} \frac{\p f}{\p \n}|}{\|\de \n^{4}+\n^{2} \frac{\p f}{\p \n}\|_{L^{s}}}\right)^{s-1}\in L^{\frac{s}{s-1}}\enn 
and derive from \eqref{q5.3}  that
 \be\la{q1}
 \begin{split}
&\left\|\de \n^{4}+\n^{2} \frac{\p f}{\p \n} \right\|_{L^{s}}\\
&\le C \left(1+\left\|\de \n^{4}+\n^{2} \frac{\p f}{\p \n}\right\|_{L^{1}}+\|\na u\|_{L^{2}}+ \|\n |u|^{2}\|_{L^{s}}+ \|\na c\|_{L^{2s}}^{2}\right).
\end{split}
\ee

Next, by  \eqref{b1}, \eqref{b2},  and   the interpolation theorem, it holds that
\be\la{q0}\ba 
\left\|\n \frac{\p f}{\p c}\right\|_{L^{\frac{6s}{3+2s}}}^{2}&\le C+C\|\n \ln \n\|_{L^{\frac{6s}{3+2s}}}^{2}\\
&\le C+C\|\n^{\g}\|_{L^{s}}^{\frac{(4s-3)}{3(\g s-1)}+\eta}\\
&\le C+C\left\|\de \n^{4}+\n^{2} \frac{\p f}{\p \n}\right\|_{L^{s}}^{\frac{(4s-3)}{3(\g s-1)}+\eta}. \ea\ee
Since $\g>\frac{4}{3}$, if    $\eta>0$ is small, one has 
$$\frac{(4s-3)}{3(\g s-1)}+\eta<1.$$ 
Utilizing \eqref{1a} and  \eqref{q0}, we obtain
\be\la{q2}\ba \|\na c  \|_{L^{2s}}^{2}&\le C\|\na^{2} c  \|_{L^{\frac{6s}{3+2s}}}^{2}  \le C\|\lap c  \|_{L^{\frac{6s}{3+2s}}}^{2} \\
 &\le C\left\|\n \frac{\p f}{\p c}\right\|_{L^{\frac{6s}{3+2s}}}^{2}+C\|\n \mu\|_{L^{\frac{6s}{3+2s}}}^{2}\\
&\le C+\frac{1}{2}\left\|\de \n^{4}+\n^{2} \frac{\p f}{\p \n}\right\|_{L^{s}}
+C\|\n \mu\|_{L^{\frac{6s}{3+2s}}}^{2}.\ea\ee
Substituting  \eqref{q2} into \eqref{q1}, we conclude \eqref{q1a}. The proof of Lemma \ref{lem3.1} is completed.
\end{proof}

\textcolor{black}{\begin{remark}Due to the    boundary condition $\frac{\p\mu}{\p n}=0$ and  the coupling  of the chemical potential  $\mu$ with the density $\n$,    the restriction   $\gamma >\frac{4}{3}$ seems   critical in  our proof especially when closing   {\it a priori} estimates on the pressure function. See  also Lemmas \ref{lem3.4}-\ref{lem3.5} below.\end{remark}}

Next,  we shall deduce some  weighted estimates on  the pressure and kinetic energy together, i.e., the weighted total energy  motivated by   \cite{jiang,pw}. 
 As in  \cite{freh}, we introduce
\be\la{r10} \xi(x)=\frac{\phi(x)\na\phi(x)}{ \left(\phi(x)+|x-x^{*}|^{\frac{2}{2-\a}}\right)^{ \a}}
\quad {\rm with}\quad x,\,x^{*}\in \overline{\om},\,\, \a\in (0,1),\ee
where the function $\phi(x)\in C^{2}(\overline{\om})$  
can be regarded as   the distance function when $x\in \om$ is close to the boundary,  smoothly extended    to the whole  domain $\om.$  In particular,    \be\la{x7}\left\{\ba& \phi(x)>0\,\,\,{\rm in}\,\,\om\,\,\,{\rm and}\,\,\,\phi(x)=0\,\,\,{\rm on}\,\,\,\p\om,\\
&|\phi(x)|\ge k_{1} \,\,\,{\rm if}\,\,\,x\in \om\,\,\,{\rm and}\,\,\,{\rm dist}(x,\,\p\om)\ge k_{2},\\
&  \na \phi=\frac{x-\tilde{x}}{\phi(x)} =\frac{x-\tilde{x}}{|x-\tilde{x}|} \,\,\,{\rm if}\,\,\,x\in \om\,\,\,{\rm and}\,\,\,{\rm dist}(x,\,\p\om)=|x-\tilde{x}|\le k_{2}, \ea\right.\ee
where the  constants $k_{i}>0$, $i=1, 2$,  are given.
See for example \cite[Exercise 1.15]{zie} for details. 

 \begin{lemma} \la{c1}  Let  $(\n,u,\mu,c)$ be the solutions stated in Proposition \ref{pro}. Then, for $\a \in (0,1)$, the following properties hold:

 {\rm (i)} In case of  $x^{*}\in \p\om,$  we have
 \be\la{r5-1}\ba  
 &\int_{B_{k_{2}}(x^{*})\cap \om} \frac{\left(\de \n^{4}+\n^{\g} +\n |u|^{2}\right)(x)}{|x-x^{*}|^{\a}}dx\\
&\quad  \le   C\left( 1+\|\na u \|_{L^{2}}+  \|\n |u|^{2}  \|_{L^{\frac{3}{2}}}+  \|\n \mu  \|_{L^{\frac{3}{2}}}^{2} \right),\ea\ee
where $k_{2}$ is taken from \eqref{x7}, and   $C$ is independent of     $x^{*}$.

{\rm (ii)} In case of $x^{*}\in  \om,$ we have
 \be\la{r5-2}\ba  
 &\int_{B_{r}(x^{*})} \frac{\left(\de \n^{4}+\n^{\g}+\n |u|^{2}\right)(x)}{|x-x^{*}|^{\a}}dx \\
 &\quad \le  C\left( 1+\|\na u \|_{L^{2}}+  \|\n |u|^{2}  \|_{L^{\frac{3}{2}}}+  \|\n \mu  \|_{L^{\frac{3}{2}}}^{2} \right),\ea\ee where $r=\frac{1}{3}{\rm dist}(x^{*},\,\p\om)>0,$ and   $C$   is  independent of $r$ or $x^{*}$.
\end{lemma}

\begin{proof}
In order to prove Lemma \ref{c1} we borrow some ideas developed in     \cite{freh,mpm,pw}  and modify the proof in \cite{liang}.

 
Write the function $f(\n,c)$ in  \eqref{b1} as
\bnn \ba f(\n,c)=  \n^{\g}+\left(H_{1}(c)+\overline{H}\right)\ln \n  +H_{2}(c) -\overline{H} \ln \n=\widetilde{f}(\n,c) -\overline{H} \ln \n,\ea\enn
where
$$\widetilde{f}(\n,c)=\n^{\g}+\left(H_{1}(c)+\overline{H}\right)\ln \n  +H_{2}(c).$$
Then, we have
\be\la{jj9} \ba   \n^{2}\frac{\p f(\n,c)}{\p \n}&=\n^{2}\frac{\p \widetilde{f}(\n,c)}{\p \n} -\n \overline{H},\ea\ee
and  \be\la{jj8} \n^{2}\frac{\p \widetilde{f}(\n,c)}{\p \n}=
 (\g-1)\n^{\g}+\n \left(H_{1}(c)+\overline{H}\right)\ge (\g-1)\n^{\g}\ge0,\ee 
 due to \eqref{b2} and \eqref{e7c}.

\smallskip

\noindent {\bf Stpe 1: Proof of  \eqref{r5-1}:} \; 
From  \eqref{r10} and \eqref{x7}  we see that
$\xi\in L^{\infty} \cap W^{1,p}_{0}$ with $p\in [2,\frac{3}{\a})$.  Furthermore,
 by  \eqref{x7}
 and the fact $\frac{2}{2-\a}>1$, one  has \be\la{z009} \phi(x)<\phi(x)+|x-x^{*}|^{\frac{2}{2-\a}} \le  C|x-x^{*}|.\ee
 With   \eqref{x7} and \eqref{z009},  one deduces that, for ${\rm dist} (x,\,\p\om)\le k_{2},$
\be\la{x8}\ba  C + \frac{C}{|x-x^{*}|^{\a}}\ge
 \div \xi(x) &\ge -C + \frac{(1-\a)}{2}\frac{|\na \phi(x)|^{2}}{\left(\phi(x)+|x-x^{*}|^{\frac{2}{2-\a}}\right)^{\a}} \\
&\ge -C + \frac{C}{|x-x^{*}|^{\a}}.\ea\ee
Thanks to  \eqref{jj9},   we multiply   $\eqref{n6}_{2}$ by    $\xi$  to obtain
 \be\la{r11}\ba&\int \left(\de \n^{4}+\n^{2} \frac{\p \widetilde{f}}{\p \n}\right)\div \xi+ \int \n  u \otimes u  :\na  \xi\\
 &= -\int \n g \cdot \xi+\int \left(\mathbb{S}_{ns}+\mathbb{S}_{c}\right):\na  \xi
 +\overline{H}\int \n \div \xi.\ea\ee
 By  \eqref{b2}, \eqref{e7b}-\eqref{e7c},   \eqref{x8}, and the fact $\xi\in L^{\infty}\cap W_{0}^{1,3}$, we estimate  the  right-hand side of \eqref{r11} as
\be\la{r13}\ba   \left|-\int \n g \cdot \xi+\int \left(\mathbb{S}_{ns}+\mathbb{S}_{c}\right):\na  \xi\right| 
& \le C(\a) \left(1+\|\na u \|_{L^{2}} +\|\na  c \|_{L^{3}}^{2}\right)\\
& \le C(\a)\left(1+\|\na u \|_{L^{2}} +\|\lap c \|_{L^{\frac{3}{2}}}^{2}\right),\ea\ee
and
\be\la{r13w}  \left|\overline{H}\int \n \div \xi\right|
 \le C  \left(1+ \int \frac{\n(x)}{|x-x^{*}|^{\a}}dx\right). \ee
For the left-hand side of \eqref{r11}, it holds from  \eqref{jj8} and  \eqref{x8} that 
\be\la{r12}\ba
 \int \left(\de \n^{4}+\n^{2} \frac{\p \widetilde{f}}{\p \n}\,\right)\div \xi\ge
 &-C \int\left(\de \n^{4}+\n^{2} \frac{\p \widetilde{f}}{\p \n}\, \right)\\
 &+C \int_{\om\cap B_{k_{2}}(x^{*})} \frac{\left(\de \n^{4}+\n^{2} \frac{\p \widetilde{f}}{\p \n}\, \right)}{|x-x^{*}|^{\a}}.\ea\ee
By  \eqref{x7}, one has $$\p_{j}\p_{i}\phi=\frac{\p_{i}(x-\tilde{x})^{j}}{\phi}-\frac{\p_{j}\phi\p_{i}\phi}{\phi}.$$
 Then,
\bnn \ba&\int \frac{\phi \n   u \otimes u \p_{j}\p_{i}\phi}{\left(\phi+|x-x^{*}|^{\frac{2}{2-\a}}\right)^{\a}} =\int \frac{\n|u |^{2}}{\left(\phi+|x-x^{*}|^{\frac{2}{2-\a}}\right)^{ \a}}-\int \frac{\n|u \cdot\na \phi|^{2}}{\left(\phi+|x-x^{*}|^{\frac{2}{2-\a}}\right)^{ \a}}.\ea\enn
Thus  we have the following computation and estimate:
\be\la{x9}\ba &\int \n u\otimes u :\na  \xi\\
&=\int \frac{  \n  |u|^{2}}{\left(\phi+|x-x^{*}|^{\frac{2}{2-\a}}\right)^{ \a}} - \a \int \frac{\phi \n(u\cdot \na\phi)^{2}}{\left(\phi+|x-x^{*}|^{\frac{2}{2-\a}}\right)^{\a+1}}\\
&\quad - \a\int \frac{\phi\n(u\cdot \na |x-x^{*}|^{\frac{2}{2-\a}})(u\cdot\na \phi)}{\left(\phi+|x-x^{*}|^{\frac{2}{2-\a}}\right)^{\a+1}}\\
&\ge (1- \a)\int \frac{  \n |u|^{2}}{\left(\phi+|x-x^{*}|^{\frac{2}{2-\a}}\right)^{\a}}
- \a\int \frac{\phi\n (u\cdot \na |x-x^{*}|^{\frac{2}{2-\a}})(u\cdot\na \phi)}{\left(\phi+|x-x^{*}|^{\frac{2}{2-\a}}\right)^{\a+1}}\\
&\ge \frac{(1- \a)}{2}\int \frac{\n |u|^{2}}{\left(\phi+|x-x^{*}|^{\frac{2}{2-\a}}\right)^{ \a}} -C \int \frac{\phi^{2}\n|u|^{2}  |x-x^{*}|^{\frac{2\a}{2-\a}}}{\left(\phi+|x-x^{*}|^{\frac{2}{2-\a}}\right)^{\a+2}}\\ &\ge C \int_{\om\cap B_{k_{2}}(x^{*})} \frac{\n|u|^{2}}{ |x-x^{*}|^{\a}} -C\|\n|u|^{2}\|_{L^{1}},\ea\ee
where   we have used   \eqref{z009}  and the Cauchy inequality.
Therefore, taking  \eqref{r13}-\eqref{x9} into account, using \eqref{q2}, \eqref{q1a},  and
  $\frac{6s}{3+2s}< \frac{3}{2}$, we deduce from \eqref{r11} that
\be\la{x10}
\ba
&\int_{\om\cap B_{k_{2}}(x^{*})} \frac{ \de \n^{4}+\n^{2} \frac{\p \widetilde{f}}{\p \n} +\n |u |^{2}}{|x-x^{*}|^{\a}}\\
& \le  C\left( \|\de \n^{4}+\n^{2} \frac{\p f}{\p \n}\|_{L^{1}}+\|\na u \|_{L^{2}}+  \|\n |u|^{2}  \|_{L^{1}}+ \|\lap c \|_{L^{\frac{3}{2}}}^{2} \right)\\
&\qquad +C \int  \frac{\n(x)}{|x-x^{*}|^{\a}}dx\\
& \le C\left( 1+\|\na u \|_{L^{2}}+  \|\n |u|^{2} \|_{L^{\frac{3}{2}}}+
  \|\n \mu  \|_{L^{\frac{3}{2}}}^{2} \right)+C\int\frac{\n(x)}{|x-x^{*}|^{\a}}dx.
  \ea\ee
Finally,  thanks to \eqref{e7c} and \eqref{jj8}, one  has
\be\la{ccc}\ba  C\int\frac{\n(x)}{|x-x^{*}|^{\a}}dx&=C\left(\int_{\om \backslash B_{k_{2}}(x^{*})}+\int_{\om\cap B_{k_{2}}(x^{*})}\right)\frac{\n(x)}{|x-x^{*}|^{\a}}dx\\
&\le C+C\int_{\om\cap B_{k_{2}}(x^{*})} \frac{\n(x)}{|x-x^{*}|^{\a}}dx\\
&\le C+\frac{1}{2}\int_{\om\cap B_{k_{2}}(x^{*})} \frac{\n^{2}
 \frac{\p \widetilde{f}}{\p \n}}{|x-x^{*}|^{\a}}.\ea\ee
Substituting  \eqref{ccc} back into \eqref{x10}, we obtain  \eqref{r5-1}.

\smallskip

\noindent {\bf Stpe 2: Proof of  \eqref{r5-2}:} \; 
Let  ${\rm dist}(x^{*},\,\p\om)=3r>0$, and   $\chi$ be  the smooth cut-off function  satisfying
 \be\la{cut} \chi(x)= 1\,\,{\rm if}\,\,x\in B_{r}(x^{*}),\quad\chi(x)=0\,\,{\rm if}\,\,
 x\notin B_{2r}(x^{*}),
 \quad |\na \chi(x)|\le 2r^{-1}.\ee 
 If we multiply $\eqref{n6}_{2}$ by
 $\frac{x-x^{*}}{|x-x^{*}|^{\a}}\chi^{2}$, we get
  \be\la{r8}\ba&\int \left(\de \n^{4} +\n^{2} \frac{\p \widetilde{f}}{\p\n}\,\right)  \frac{3-\a}{|x-x^{*}|^{\a}}\chi^{2} + \int \n u\otimes u :\na  \left(\frac{x-x^{*}}{|x-x^{*}|^{\a}}\chi^{2}\right)\\
&= -\int  \n g \cdot \frac{x-x^{*}}{|x-x^{*}|^{\a}}\chi^{2}+\int \left(\mathbb{S}_{ns}+\mathbb{S}_{c}\right):\na  \left(\frac{x-x^{*}}{|x-x^{*}|^{\a}}\chi^{2}\right)\\
&\quad
 -2\int \left(\de \n^{4} +\n^{2} \frac{\p f}{\p \n}\right)  \chi \frac{\na\chi\cdot(x-x^{*})}{|x-x^{*}|^{\a}}+\overline{H}\int \n \frac{3-\a}{|x-x^{*}|^{\a}}\chi^{2}.\ea\ee
From the following computation,
\bnn\ba &\p_{i} \left(\frac{x^{j}-(x^{*})^{j}}{|x-x^{*}|^{\a}}\chi^{2}\right)\\
&= \frac{\p_{i}(x^{j}-(x^{*})^{j})}{|x-x^{*}|^{\a}}\chi^{2} -\a \frac{(x^{j}-(x^{*})^{j})(x^{i}-(x^{*})^{i})}{|x-x^{*}|^{\a+2}}\chi^{2}+2\chi \frac{x^{j}-(x^{*})^{j}}{|x-x^{*}|^{\a}}\p_{i} \chi,
\ea\enn
 one sees that the second term on the left-hand side of \eqref{r8} satisfies
 \be \ba \label{325}
 &\int \n u\otimes u:\na \left(\frac{x-x^{*}}{|x-x^{*}|^{\a}}\chi^{2}\right)\\
 &\ge(1-\a)\int \frac{\n |u|^{2}}{|x-x^{*}|^{\a}}\chi^{2}+2\int  \frac{ \chi\n (u\cdot\na \chi )(u\cdot(x-x^{*}))}{|x-x^{*}|^{\a}}\\
 &\ge\frac{1-\a}{2}\int \frac{\n |u|^{2}}{|x-x^{*}|^{\a}}\chi^{2}-C \int_{B_{2r}(x^{*})\backslash B_{r}(x^{*})} \frac{\n|u|^{2}}{|x-x^{*}|^{\a}},\ea\ee
 where the constant $C$ is independent of $r,$ and for the last inequality   we have used $|\na \chi||x-x^{*}|\le 4$ for any $x\in B_{2r}(x^{*})\backslash B_{r}(x^{*}).$
Owing to \eqref{cut}, \eqref{q2}, and  the fact $$\na  \left(\frac{x-x^{*}}{|x-x^{*}|^{\a}}\chi^{2}\right) \in L^{3},$$  we have the following estimates:
 \be \ba \label{326}
 &\left| -\int   \n g  \cdot \frac{x-x^{*}}{|x-x^{*}|^{\a}}\chi^{2}+\int \left( \mathbb{S}_{ns}+\mathbb{S}_{c}\right) :\na  \left(\frac{x-x^{*}}{|x-x^{*}|^{\a}}\chi^{2}\right) \right|\\
 &\le C  \left(1+\|\na u\|_{L^{2}} + \|\lap c\|_{L^{\frac{3}{2}}}^{2}\right), \ea\ee
 and
\be\ba\label{327}
&\left|-2\int \left(\de \n^{4} +\n^{2} \frac{\p f}{\p \n}\right)  \chi \frac{\na\chi\cdot(x-x^{*})}{|x-x^{*}|^{\a}}+\overline{H}\int \n \frac{3-\a}{|x-x^{*}|^{\a}}\chi^{2}\right|\\
&\le   C\int_{B_{2r}(x^{*})\backslash B_{r}(x^{*})}
 \frac{\left(\de \n^{4} +\n^{2}\frac{\p \widetilde{f}}{\p\n}\right)}{|x-x^{*}|^{\a}}+C\int_{B_{2r}(x^{*})} \frac{\n(x)}{|x-x^{*}|^{\a}}dx,\ea \ee
where $C$ is independent of $r.$

With the above three estimates \eqref{325}-\eqref{327} in hand, we deduce from  \eqref{r8} that
  \be\ba \label{328}
 & \int_{B_{r}(x^{*})} \frac{\left(\de \n^{4} +\n^{2} \frac{\p \widetilde{f}}{\p \n}+\n|u|^{2}\right)(x)}{ |x-x^{*}|^{\a}}dx \\
& \le C\left(1  +  \|\na u\|_{L^{2}} + \|\lap c\|_{L^{\frac{3}{2}}}^{2} \right)\\
 &\quad +  C \int_{B_{2r}(x^{*})\backslash B_{r}(x^{*})} \frac{\left(\de \n^{4} +\n^{2} \frac{\p \widetilde{f}}{\p \n}+\n|u|^{2}\right)(x)} {|x-x^{*}|^{\a}}dx\\
&\quad+C\int_{B_{2r}(x^{*})} \frac{\n(x)}{|x-x^{*}|^{\a}}dx.\ea\ee
By  \eqref{x10}-\eqref{ccc} and the following estimate
\bnn \ba C\int_{B_{2r}(x^{*})} \frac{\n(x)}{|x-x^{*}|^{\a}}dx&=C\left(\int_{B_{r}(x^{*})}+\int_{B_{2r}(x^{*})\backslash B_{r}(x^{*})}\right) \frac{\n(x)}{|x-x^{*}|^{\a}}dx\\
&\le C+\frac{1}{2}\int\frac{\n^{2} \frac{\p \widetilde{f}}{\p \n} (x)}{|x-x^{*}|^{\a}} +
  C \int_{B_{2r}(x^{*})\backslash B_{r}(x^{*})} \frac{\n^{2} \frac{\p \widetilde{f}}{\p \n} (x)} {|x-x^{*}|^{\a}}dx, \ea\enn
 we obtain from \eqref{328} that
    \be\la{r9*}\ba &\int_{B_{r}(x^{*})} \frac{\left(\de \n^{4} +\n^{2} \frac{\p \widetilde{f}}{\p \n}+\n|u|^{2}\right)(x)}{ |x-x^{*}|^{\a}}dx \\
& \le C\left(1  +  \|\na u\|_{L^{2}} + \|\n|u|^{2}\|_{L^{\frac{3}{2}}}
+\|\n  \mu\|_{L^{\frac{3}{2}}}^{2}\right) \\
&\quad+
 C \int_{B_{2r}(x^{*})\backslash B_{r}(x^{*})} \frac{\left(\de \n^{4} +\n^{2} \frac{\p \widetilde{f}}{\p \n}+\n|u|^{2}\right)(x)} {|x-x^{*}|^{\a}}dx.\ea\ee

It remains   to deal with   the last term in \eqref{r9*}. To this end, we use the ideas developed in \cite{liang} and   divide the proof into two cases: $(1)$  $x^{*}\in \om$ is far away from the boundary;  $(2)$  $x^{*}\in \om$ is close to the boundary.

$(1)$  For the case of  ${\rm dist}(x^{*},\,\p\om)= 3r\ge \frac{k_{2}}{2}>0$  with  $k_{2}$ being taken from  \eqref{x7}, 
it is clear that
 \be\la{r9a1}\ba   \int_{B_{2r}(x^{*})\backslash B_{r}(x^{*})} \frac{\left(\de \n^{4}
  +\n^{2} \frac{\p\widetilde{f}}{\p \n}+\n|u|^{2}\right)(x)} {|x-x^{*}|^{\a}}dx \le C(k_{2})   \left\|\de \n^{4} +\n^{2} \frac{\p\widetilde{f}}{\p \n}+\n|u|^{2}\right\|_{L^{1}}.\ea\ee
With \eqref{r9a1}, as well as \eqref{jj9}-\eqref{jj8}, \eqref{ccc},   Lemma \ref{lem3.1}, we deduce from \eqref{r9*} that
 \be\la{r9a}\ba &\int_{B_{r}(x^{*})} \frac{\left(\de \n^{4} +\n^{\gamma}  +\n|u|^{2}\right)(x)}{ |x-x^{*}|^{\a}}dx\\
 &\le C\int_{B_{r}(x^{*})} \frac{\left(\de \n^{4} +\n^{2}  \frac{\p\widetilde{f}}{\p \n}+\n|u|^{2}\right)(x)}{ |x-x^{*}|^{\a}}dx\\
 & \le C\left(1+  \|\na u \|_{L^{2}} + \|\n  |u|^{2}\|_{L^{\frac{3}{2}}} +\|\n  \mu\|_{L^{\frac{3}{2}}}^{2}+   \|\de \n^{4} +\n^{2}\frac{\p\widetilde{f}}{\p \n}+\n|u|^{2}\|_{L^{1}}\right)\\
 & \le C\left( 1+\|\na u \|_{L^{2}}+  \|\n |u|^{2}  \|_{L^{\frac{3}{2}}}+  \|\n \mu  \|_{L^{\frac{3}{2}}}^{2} \right).\ea\ee

$(2)$  For the case of $x^{*}\in \om$ close to the boundary,  that is, 
${\rm dist}(x^{*},\,\p\om)= 3r<\frac{k_{2}}{2}$, 
let $ |x^{*}-\tilde{x}^{*}|={\rm dist}(x^{*},\,\p\om)$ with $\tilde{x}^{*}\in \p\om.$  Then, one deduces (see Figure 1 below) that
       \be\la{bb1} 4 |x-x^{*}|\ge |x-\tilde{x}^{*}|,\quad \forall \,\,\, x\notin B_{r}(x^{*}).\ee
\begin{figure}[H]
\centering
\includegraphics[width=0.5\textwidth]{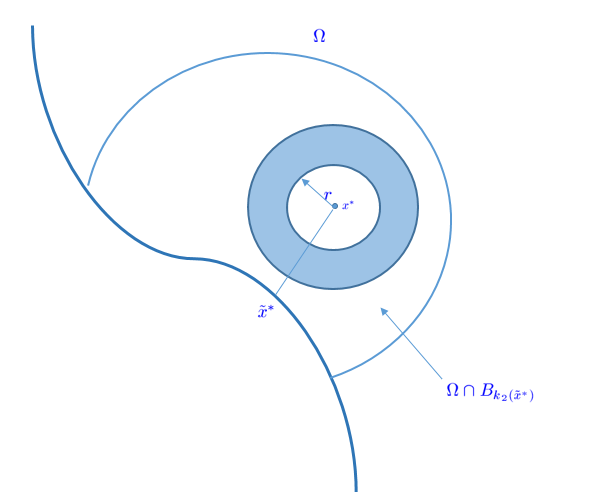}\hspace{-0.5cm}
\caption{Near boundary points}
\end{figure}
Making use of \eqref{bb1}  and  \eqref{x10}-\eqref{ccc},  we have the following estimate, 
\be\ba \label{333} 
& C \int_{B_{2r}(x^{*})\backslash B_{r}(x^{*})} \frac{\left(\de \n^{4}
  +\n^{2} \frac{\p \widetilde{f}}{\p \n}+\n|u|^{2}\right)(x)} {|x-x^{*}|^{\a}}dx\\
  & \le   C \int_{\om\cap B_{k_{2}}(\tilde{x}^{*})}\frac{\left(\de \n^{4}
   +\n^{2} \frac{\p \widetilde{f}}{\p \n}+\n|u|^{2}\right)(x)} {|x-\tilde{x}^{*}|^{\a}}dx\\
& \le C\left( 1+\|\na u\|_{L^{2}}+  \|\n |u|^{2}\|_{L^{\frac{3}{2}}}+  \|\n \mu  \|_{L^{\frac{3}{2}}}^{2} \right).\ea\ee
This inequality  \eqref{333} and  \eqref{jj9}-\eqref{jj8} ensure that  \eqref{r9*} leads to
  \be\la{r9b}\ba  &\int_{B_{r}(x^{*})} \frac{\left(\de \n^{4} +\n^{\gamma}  +\n|u|^{2}\right)(x)}{ |x-x^{*}|^{\a}}dx\\
 &\le C\int_{B_{r}(x^{*})} \frac{\left(\de \n^{4} +\n^{2}  \frac{\p\widetilde{f}}{\p \n}+\n|u|^{2}\right)(x)}{ |x-x^{*}|^{\a}}dx\\
& \le C\left( 1+\|\na u\|_{L^{2}}+  \|\n |u|^{2}  \|_{L^{\frac{3}{2}}}+  \|\n \mu  \|_{L^{\frac{3}{2}}}^{2} \right).\ea\ee
Therefore, the  desired estimate \eqref{r5-2} follows immediately from \eqref{r9a} and \eqref{r9b}.
The proof of Lemma \ref{c1} is completed.
\end{proof}

The next lemma provides  a refined estimate on the weighted energy  
obtained in Lemma \ref{c1}.
 \begin{lemma} \la{c2} Let the  assumptions    in Lemma \ref{c1} hold true. Assume that  there is a   constant $\mathbf{M}$   uniform  in $\de$, such that
\be\la{y1}\mathbf{M}= \max\{1, \|\n\|_{L^{2}}\}<\infty.\ee Then,
   \be\la{r5}\ba  \sup_{x^{*}\in \overline{\om}}\int_{\om} \frac{\left(\de \n^{4} +\n^{\g}  +\n|u|^{2}\right)(x)}{|x-x^{*}|^{\a^{2}}}dx\le  C\left( 1+\|\na u \|_{L^{2}}+  \|\n |u|^{2}  \|_{L^{\frac{3}{2}}}+  \|\na \mu  \|_{L^{2}}^{2}\right).\ea\ee
\end{lemma}

\begin{proof}
If $x^{*}\in \p\om,$  it holds that,  for  any $r\in (0,k_{2})$, 
\be\la{r5ar}\ba& \frac{1}{r^{\a(1-\a)}}\int_{B_{r}(x^{*})\cap \om} \frac{\left(\de \n^{4} +\n^{\gamma}+\n |u|^{2}\right)(x)}{|x-x^{*}|^{\a^{2}}}dx\\
 &\le  \int_{B_{r}(x^{*})\cap \om} \frac{\left(\de \n^{4} +\n^{\gamma}+\n |u|^{2}\right)(x)}{|x-x^{*}|^{\a}}dx\\
&\le \int_{B_{k_{2}}(x^{*})\cap \om} \frac{\left(\de \n^{4} +\n^{\gamma}
+\n |u|^{2}\right)(x)}{|x-x^{*}|^{\a}}dx.\ea\ee
Combining \eqref{r5ar} with   \eqref{r5-1}, we obtain for any $r\in (0,k_{2})$,
\be\la{r5a}\ba & \int_{B_{r}(x^{*})\cap \om} \frac{\left(\de \n^{4} +\n^{\gamma} +\n |u|^{2}\right)(x)}{|x-x^{*}|^{\a^{2}}}dx\\
&\le   C r^{\a(1-\a)}\left( 1+\|\na u \|_{L^{2}}+  \|\n |u|^{2}  \|_{L^{\frac{3}{2}}}+  \|\n \mu  \|_{L^{\frac{3}{2}}}^{2} \right), \ea\ee
where the constant  $C$ is independent of $r$ or $x^{*}$.
Using \eqref{e7b}, \eqref{3.30}, and the interpolation inequality, we have the following estimate:
\be\la{y6}\ba\|\n \mu  \|_{L^{\frac{3}{2}}}^{2} &\le C\|\n\|_{L^{2}}^{2}\|\mu\|_{L^{6}}^{2}\\
&\le C\|\n\|_{L^{2}}^{2}\left(1+\|\n\|_{L^{\frac{6}{5}}} \|\na\mu\|_{L^{2}}\right)^{2}\\
&\le C\textbf{M}^{\frac{8}{3}}\left(1+\|\na \mu\|_{L^{2}}^{2}\right),\ea\ee
where, and in what follows,  the constant $C$ is independent of $\textbf{M}$.
Choosing   $r_{0}$    small so that
\be\la{y7} r_{0} \le \min\left\{\frac{k_{2}}{2},\, \textbf{M}^{\frac{-8}{3\a(1-\a)}}\right\}.\ee
It follows from  \eqref{r5a} that
\be\la{r5a1}\ba  &\int_{B_{r_{0}}(x^{*})\cap \om} \frac{\left(\de \n^{4} +\n^{\gamma}+\n |u|^{2}\right)(x)}{|x-x^{*}|^{\a^{2}}}dx\\
&\le   C \left( 1+\|\na u \|_{L^{2}}+  \|\n |u|^{2}  \|_{L^{\frac{3}{2}}} \right)
+C r_{0}^{\a(1-\a)}\|\n \mu  \|_{L^{\frac{3}{2}}}^{2}\\
&\le   C \left( 1+\|\na u \|_{L^{2}}+  \|\n |u|^{2}  \|_{L^{\frac{3}{2}}} \right)
+C r_{0}^{\a(1-\a)}\textbf{M}^{\frac{8}{3}}\left(1+\|\na \mu\|_{L^{2}}^{2}\right)\\
&\le   C \left( 1+\|\na u \|_{L^{2}}+  \|\n |u|^{2}  \|_{L^{\frac{3}{2}}}+
 \|\na  \mu \|_{L^{2}}\right),\ea\ee where, for  the last two inequalities, we have used
  \eqref{y6} and \eqref{y7}.

If  $x^{*}\in  \om,$  we use the  similar arguments  to obtain
 \be\la{r5b1}\ba  &\int_{B_{r_{0}}(x^{*})} \frac{\left(\de \n^{4} +\n^{\gamma}+\n |u|^{2}\right)(x)}{|x-x^{*}|^{\a^{2}}}dx\\
&\le C r_{0}^{\a(1-\a)}\int_{B_{r_{0}}(x^{*})} \frac{\left(\de \n^{4} +\n^{\gamma}+\n |u|^{2}\right)(x)}{|x-x^{*}|^{\a}}dx\\
&\le C r_{0}^{\a(1-\a)}\left( 1+\|\na u \|_{L^{2}}+  \|\n |u|^{2}  \|_{L^{\frac{3}{2}}}+ \|\n \mu  \|_{L^{\frac{3}{2}}}^{2} \right)\\
&\le C \left( 1+\|\na u \|_{L^{2}}+  \|\n |u|^{2}  \|_{L^{\frac{3}{2}}}+  \|\na \mu  \|_{L^{2}}^{2} \right).\ea\ee
  As a result of  \eqref{r5a1} and \eqref{r5b1}, we conclude   \eqref{r5} by using the Finite Coverage Theorem,  as  the domain
$\overline{\om}$ is bounded.  The proof of Lemma \ref{c2} is completed.
\end{proof}

The final two    Lemmas  \ref{lem3.4} and   \ref{lem3.5} are devoted to  proving the desired  inequality \eqref{y16} and the {\it a priori} bound \eqref{y1}.

\begin{lemma}\la{lem3.4} Let the assumptions in Proposition \ref{thm1.2} hold true.
 Then \be\ba\la{q1ax}\|u \|_{H_{0}^{1}} + \|\na\mu\|_{L^{2}} \le C .\ea\ee
 \end{lemma}

\begin{proof}
Define
\be\la{y12} \mathbb{A}=\int \n |u|^{2}|u|^{2(1-\th)}\quad {\rm with}\quad \th=\frac{3\g-4}{8\g}.\ee
By  \eqref{y10},  one  has 
\be  \la{y13}
 \th \in (0,\,   \frac{1}{8}].
\ee 
Thanks to  \eqref{e7b}
and the  H\"{o}lder inequality,  it holds that
\be\la{y8} \|\n u \|_{L^{1}}\le \|\n  |u|^{2}|u|^{2(1-\th)}\|_{L^{1}}^{\frac{1}{2(2-\th)}}\|\n \|_{L^{1}}^{\frac{3-2\th}{2(2-\th)}}\le C \mathbb{A}^{\frac{1}{2(2-\th)}} \ee
 and
  \be\la{y9} \|\n |u|^{2}\|_{L^{\frac{3}{2}}}\le
   \|\n  |u|^{2}|u|^{2(1-\th)}\|_{L^{1}}^{\frac{1}{2-\th}}\|\n\|_{L^{1}}^{\frac{(1-\th)}{2-\th}}\le C\mathbb{A}^{\frac{1}{2-\th}}.\ee
By means of \eqref{01},  \eqref{b2},   \eqref{e7a},    \eqref{y8},  we get
\be\la{r1s}\ba  \int \left(|\na u|^{2} +   |\na \mu|^{2}\right)
  \le  C\|\n u\|_{L^{1}} \le C \mathbb{A}^{\frac{1}{2(2-\th)}}.\ea\ee

Let
   \be\la{y14}   \a^{2}=1- \frac{\th}{2}.\ee
 One calculates  as the following,
 \be\la{y11}\ba  \frac{\n  |u|^{2(1-\th)}}{|x-x^{*}|}
   &=\left(\frac{\n |u|^{2}}{|x-x^{*}|^{\a^{2}}}\right)^{1-\th} \left(\frac{\n^{\g}}{|x-x^{*}|^{\a^{2}}}\right)^{\frac{\th}{\g}}
 \left(\frac{1}{|x-x^{*}|^{\frac{\g}{2(\g-1)}+\a^{2}}}\right)^{\frac{(\g-1)\th}{\g}},\ea\ee
 where $\frac{\g}{2(\g-1)}+\a^{2}<3$ since $\gamma>\frac{4}{3}.$ 
Hence, utilizing  \eqref{r5},    \eqref{y9}, \eqref{r1s},  we  integrate \eqref{y11} and obtain
\be\la{y20}\ba \int \frac{\n  |u|^{2(1-\th)}(x)}{|x-x^{*}|}dx
 &\le \int\frac{\n |u|^{2}(x)}{|x-x^{*}|^{\a^{2}}}dx+\int\frac{\n^{\g}(x)}{|x-x^{*}|^{\a^{2}}}dx+C\\
&\le C\int_{\om} \frac{\left(\de \n^{4} +\n^{\gamma} +\n|u|^{2}\right)(x)}{|x-x^{*}|^{\a^{2}}}dx\\
&\le  C\left( 1+\|\na u \|_{L^{2}}+  \|\n |u|^{2}  \|_{L^{\frac{3}{2}}}+  \|\na \mu  \|_{L^{2}}^{2}\right)\\
&\le C\left(1+ \mathbb{A}^{\frac{1}{2-\th}}\right).\ea\ee
From  \eqref{r1s},   \eqref{y20},  and Part (i) in Lemma \ref{lem2.3}, one deduces
\bnn\ba  \mathbb{A}&\le \|\na u\|_{L^{2}}^{2}\sup_{x^{*}\in \overline{\om}}
\int \frac{\n  |u|^{2(1-\th)}(x)}{|x-x^{*}|}dx\\&
\le C\mathbb{A}^{\frac{1}{2(2-\th)}}\left(1+ \mathbb{A}^{\frac{1}{2-\th}}\right)\\&
 \le 1+C\mathbb{A}^{\frac{3}{2(2-\th)}},\ea\enn
which together  with \eqref{y13} yields    \be\la{y15} \mathbb{A}\le C.\ee
Combining  \eqref{y15} with \eqref{r1s}, we get \eqref{q1ax}. The proof of Lemma \ref{lem3.4} is completed.
\end{proof}

\begin{lemma}\la{lem3.5} 
Let the assumptions in Theorem \ref{thm1.2} hold true. Then, \be\ba\la{q1as} \|\de \n^{4}+\n^{\g} \|_{L^{s}} +\|\mu\|_{
L^{6}} +\|c\|_{
W_{n}^{2,\frac{3}{2}}}\le C .\ea\ee
\end{lemma}

\begin{proof}
Owing to \eqref{y12} and \eqref{y14}, one has  \bnn \frac{3\g-4\a^{2}}{3\g-4}\in (0,3).\enn
By  \eqref{y10}, \eqref{y20},  \eqref{y15}, and   the H\"older inequality, we have the following estimate,
  \be\la{z2}\ba
\int \frac{\n^{\frac{4}{3}}(x)}{|x-x^{*}|}dx
&\le \left(\int \frac{\n^{\g}(x)}{|x-x^{*}|^{\a^{2}}}dx\right)^{\frac{4}{3\g}}\left(\int \frac{dx}{|x-x^{*}|^{\frac{3\g-4\a^{2}}{3\g-4}}}\right)^{\frac{3\g-4}{3\g}}\\
&\le C \left(\int \frac{\n^{\g}(x)}{|x-x^{*}|^{\a^{2}}}dx\right)^{\frac{4}{3\g}}\\
&\le C.\ea\ee
Hence, using  \eqref{r1s}, \eqref{y15},   \eqref{z2},  Part (ii)  in  Lemma \ref{lem2.3}, we find
  \be\la{ee1}\|\left(\n^{\frac{4}{3}}-(\n^{\frac{4}{3}})_{\om}\right) \mu^{2}\|_{L^{1}}\le \|\na \mu\|_{L^{2}}^{2}\left(1+\sup_{x^{*}}
  \int \frac{\n^{\frac{4}{3}}(x)}{|x-x^{*}|}dx\right)\le C.\ee
   On the other hand, it follows from \eqref{e7b}, \eqref{z2}, Lemma \ref{lem2.1}, Lemma \ref{lem3.4}, and  the interpolation inequality that
  \be\la{ee2} \|(\n^{\frac{4}{3}})_{\om}\mu^{2}\|_{L^{1}}\le (\n^{\frac{4}{3}})_{\om}\|\mu^{2}\|_{L^{1}} \le  \| \mu \|_{L^{2}}^{2}\le C\|\n\|_{L^{2}}^{\frac{2}{3}}.\ee Therefore,  utilizing  \eqref{ee1}-\eqref{ee2}  and the fact $\frac{6s}{3+2s}\le \frac{3}{2}$, we conclude
\be\la{y22}\ba \|\n \mu\|_{L^{\frac{6s}{3+2s}}}^{2}&\le C\|\n \mu\|_{L^{\frac{3}{2}}}^{2}\\
&\le C\|\n^{\frac{4}{3}}\mu^{2}\|_{L^{1}}\|\n \|_{L^{2}}^{\frac{2}{3}}\\
&\le C \left(\left\|\left(\n^{\frac{4}{3}}-(\n^{\frac{4}{3}})_{\om}\right) \mu^{2}\right\|_{L^{1}}+\left\|(\n^{\frac{4}{3}})_{\om}\mu^{2}\right\|_{L^{1}}\right)\|\n \|_{L^{2}}^{\frac{2}{3}}\\& \le C \|\n \|_{L^{2}}^{\frac{4}{3}}.\ea\ee
Substituting \eqref{y22}   into  \eqref{q1a},    using \eqref{y9}, \eqref{r1s}, \eqref{y15},  we get
\be\la{y23}\ba   \|\de \n^{4}+\n^{\g}\|_{L^{s}}
 &\le \left\|\de \n^{4}+\n^{2} \frac{\p f}{\p \n}\right\|_{L^{s}}\\
  &\le C \left(1 +\|\na u\|_{L^{2}}+ \|\n |u|^{2}\|_{L^{s}}+ \|\n \mu\|_{L^{\frac{6s}{3+2s}}}^{2}\right)\\
  &\le C \left(1 +\|\na u\|_{L^{2}}+ \|\n |u|^{2}\|_{L^{\frac{3}{2}}}+ \|\n \mu\|_{L^{\frac{3}{2}}}^{2}\right)\\
  &\le C\left(1 +\|\n \|_{L^{2}}^{\frac{4}{3}}\right).\ea\ee
Thanks to   \eqref{y10}, one has
\be\la{y24}\ba  C\|\n \|_{L^{2}}^{\frac{4}{3}}
\le C \|\n^{\g} \|_{L^{\frac{2}{\g}}}^{\frac{4}{3\g}}\le C +\frac{1}{2}\|\n^{\g}\|_{L^{s}} \le C +\frac{1}{2}\|\de \n^{4}+\n^{\g}\|_{L^{s}}.\ea\ee
The combination of  \eqref{y23} with \eqref{y24} gives rise  to
\be\la{y25}\ba   \|\de \n^{4}+\n^{\g}\|_{L^{s}} \le \overline{C},\quad (\g s>2),\ea\ee
where  $\overline{C}$ depends only on $m_{1},\g,\overline{H},\lam_{1},\lam_{2},
|\om|,\|g\|_{L^{\infty}}$. 
From \eqref{y25}, there is a constant $C_0$ independent of $\delta$, such that
$$\|\n\|_{L^{2}}\le C_0,$$
and hence  we are allowed to
 select  in \eqref{y1}
\be\la{ee3} \textbf{M}=2C_0\ee
and close the   {\it a priori}  assumption in \eqref{y1}.

It only remains   to derive the bound  of $\|c\|_{W_{n}^{2,\frac{3}{2}}}.$  From  \eqref{b1},  \eqref{y22}, \eqref{y25}, it follows that
 \be\la{y26}\ba \|\na^{2} c  \|_{L^{\frac{3}{2}}}
 &\le C\|\lap c\|_{L^{\frac{3}{2}}}  \\
 &\le C\left\|\n \frac{\p f}{\p c}\right\|_{L^{\frac{3}{2}}} +C\|\n \mu\|_{L^{\frac{3}{2}}} \\
&\le C. \ea\ee
From  \eqref{e7c} and \eqref{y25},  the same argument as   \eqref{bb12} yields 
 \bnn\ba
\int c=\frac{|\om|}{m_{1}}\int \n\left(c\right)_{\om}
&=\frac{|\om|}{m_{1}}\int \n c-\frac{|\om|}{m_{1}}\int \n\left(c-\left(c\right)_{\om}\right)\\
&=\frac{|\om|m_{2}}{m_{1}} -\frac{|\om|}{m_{1}}\int \n\left(c-\left(c\right)_{\om}\right)\\
&\le C  +C \|\n\|_{L^{\frac{6}{5}}}\|\na c\|_{L^{2}}\\
&\le C  +C \|\na c\|_{L^{2}},\ea\enn
which  implies
\be\ba \label{363}
 \|c\|_{L^{1}}&\le  \|c-\left(c\right)_{\om}\|_{L^{1}} + \|\left(c\right)_{\om}\|_{L^{1}}\le C
 +C \|\na c\|_{L^{2}}.\ea\ee
Then \eqref{363} and \eqref{y26} provide us the following estimate:
 \be\la{y30} \|c\|_{W_{n}^{2,\frac{3}{2}}}\le C.\ee
 In conclusion, the desired estimate \eqref{q1as} follows from \eqref{3.30}, \eqref{r1s},  \eqref{y15}, \eqref{y25}, and \eqref{y30}.
The proof of Lemma \ref{lem3.5} is completed.
\end{proof}

 Therefore,  the proof  of Proposition \ref{thm1.2} and  hence Theorem \ref{thm1.1} is completed.

\bigskip

\section*{Acknowledgement}
The research  of D. Wang was partially supported by the National Science Foundation under grant  DMS-1907519.
The authors would like to thank the anonymous referees for valuable comments
and suggestions.

\end{document}